\documentclass[a4paper]{paper}
\usepackage[utf8]{inputenc}
\usepackage{amsthm,amsfonts}
\usepackage[sumlimits]{amsmath}
\usepackage{amssymb}
\usepackage[all]{xy}
\usepackage{tikz}
\usepackage{color}

\newtheorem{lema}{Lemma}[section]
\newtheorem{teo}[lema]{Theorem}

\theoremstyle{definition}

\numberwithin{equation}{section}

\newcommand{\tq}{\,|\,}

\newcommand{\real}{\mathbb{R}}
\newcommand{\z}{\mathbb{Z}}

\newcommand{\apres}[2]{\langle #1 \tq #2\rangle}

\newcommand{\ed}{\ar@{-}}
\newcommand{\onto}{\twoheadrightarrow}

\renewcommand{\epsilon}{\varepsilon}

\newcommand{\pt}{\,\,\forall}
\newcommand{\xtil}{\tilde{x}}

\newcommand{\xhat}{\hat{x}}
\newcommand{\vhat}{\hat{v}}

\newcommand{\ehat}{\hat{e}}

\newcommand{\ggoth}{\mathcal{G}}
\newcommand{\lgoth}{\mathcal{L}}
\newcommand{\fgoth}{\mathcal{F}}
\newcommand{\hgoth}{\mathcal{H}}
\newcommand{\lf}{\lgoth_{\fgoth}}

\numberwithin{equation}{section}

\newcommand{\BZ}{\mathbb{Z}}

\DeclareMathOperator{\Id}{Id}
\DeclareMathOperator{\Hom}{Hom}

\DeclareMathOperator{\lcm}{lcm}

\begin{document}

\title{The $\Sigma^1$ invariant for some Artin groups of arbitrary circuit rank}

\author{Kisnney Almeida \\ \\
Departamento de Ciências Exatas, \\ Universidade Estadual de Feira de Santana (UEFS), \\ 44031-460 Feira de Santana, BA, Brazil;
}

\maketitle

\begin{abstract} We classify the Bieri-Neumann-Strebel-Renz invariant $\Sigma^1(G)$ for a class of Artin groups with minimal graphs of arbitrary circuit rank.
\end{abstract}

\section{Introduction}

The class of groups $G$ of homotopical type $F_m$  was first studied by C. T. C. Wall  and later a homological version, i.e. the class of  groups of type $FP_m$, was introduced by Bieri and Eckmann \cite{Bi}. By definition a group $G$ is of type $FP_m$ if the trivial $\mathbb{Z} G$-module $\mathbb{Z}$ has a projective resolution with all projectives finitely generated in dimensions $\leq n$ and $G$ is of type $F_m$ if there is a $K(G,1)$ CW-complex with finite $m$ skeleton. The homological and homotopical $\Sigma$-invariants of a group $G$ are monoid versions of the types $F_m$ and $FP_m$ for some special submonoids of $G$ associated to non-zero homomorphisms from $G$ to $\mathbb{R}$.

The first $\Sigma$-invariant was introduced originally by Bieri and Strebel for the class of metabelian groups in \cite{BiSt} and  was used to classify all finitely presented metabelian groups. This definition was later extended by Bieri, Neumann and Strebel  to all  finitely generated groups \cite{BiNSt}. New homological invariants $\{ \Sigma^m(G, \z) \}_{m \geq 1}$ were defined by Bieri and Renz in \cite{BiRe} and a series of homotopical $\Sigma$-invariants $\{ \Sigma^m(G) \}_{m \geq 1}$  was studied by Renz \cite{Re} and Meinert \cite{Mn}. It is interesting to note that  the invariant $\Sigma^m(G, \z)$ (resp. $\Sigma^m(G)$)  classifies which subgroups of $G$ above the commutator have homological type $FP_m$ (resp. $F_m$) provided the group $G$ is of type $FP_m$ (resp. $F_m$) \cite{BiRe}. As shown by Bieri and Groves in the case of metabelian groups, $\Sigma^1(G)$ has deep links with valuation theory from commutative algebra \cite{BiGr}.

The $\Sigma$-invariants are calculated for very few groups in all dimensions $m$, for example the case of the R. Thompson group $F$ was treated by Bieri, Geoghegan and Kochloukova in \cite{BiGeKo} but the case of the generalised Thompson groups $F_{n, \infty}$ is known only in dimension $m = 2$ for $n \geq 3$ \cite{Ko}. The case of metabelian groups of finite Prüfer rank was solved completely by Meinert \cite{Mn}.

Let $\mathcal{G}$ be a finite simplicial graph  with edges labeled by integer numbers greater than one. The Artin group $G$ associated to $\mathcal{G}$ has a finite presentation, with generators  the vertices $V(\mathcal{G})$  of $\mathcal{G}$ and relations
    $$\underbrace{uvu\cdots}_{n \text{ factors}}=\underbrace{vuv\cdots}_{n \text{ factors}} \hbox{ for each edge } \hbox{ of }\mathcal{G} \hbox{ with vertices } u,v \hbox{ and label }n \geq 2.$$
A subgraph $\hgoth$ of $\ggoth$ is called \textbf{dominant} if, for each $v\in V(\ggoth) \setminus V(\hgoth)$, there is an edge $e\in E(\ggoth)$ such that $\sigma(e)=v$ and $\tau(e)\in V(\hgoth)$, where $\sigma(e)$ is the beginning and $\tau(e)$ is the end of $e$.

Let $\chi : G \to \real $ be a non-zero character of an Artin group $G$ with underlying graph $\mathcal{G}$. An edge $e$ of $\mathcal{G}$ is called \textbf{dead} if $e$ has an even label greater than 2 and $\chi(\sigma (e))=-\chi(\tau (e))$.
 Set $\mathcal{L}_{\mathcal{F}}=\lf(\chi)$ as the complete subgraph of $\mathcal{G}$ generated by the vertices $v\in V(\mathcal{G})$ such that $\chi(v)\neq 0$. Define the \textbf{living subgraph} $\mathcal{L}=\mathcal{L}(\chi)\subset \mathcal{L}_{\mathcal{F}}$ as the subgraph obtained from $\mathcal{L}_{\mathcal{F}}$ after removing the dead edges.

If $\ggoth$ is connected the fundamental group
    $\pi_1(\ggoth)$ is free of rank $n$ for some $n \geq 0$. In this case we say that $G$ is an Artin group of \textbf{circuit rank} $n$. It was shown in  \cite{AlKo} that for Artin groups $G$ of circuit rank 1 $\Sigma^1(G)=\{[\chi]\in S(G)\tq \mathcal{L}(\chi) \text{ is a connected dominant subgraph of $\mathcal{G}$} \}.$ The case of Artin groups G of circuit rank 2 will be similarly treated in \cite{Al}.  To prove these results, we use some strategies to reduce the general graph to as few minimal graphs as possible with the same relevant properties. However, even these minimal cases are very hard to work with as the circuit rank of the graph increases. For instance we have treated the case of circuit rank 3 with the full graph on 4 vertices as underlying graph in \cite{AlKo2}, which was only possible by applying major restrictions on the labels.

    In this paper\footnote{I thank Dessislava Kochloukova for the valuable ideas and advisoring.} we prove the description of $\Sigma^1$ for a class of Artin groups whose underlying graphs are minimal graphs of arbitrary rank. This could be used, together with the techniques we show in \cite{AlKo}, to validate the same description to a much larger class of groups. In fact, we intend to use a special case ($n=3$) of this result to prove the circuit rank 2 case in \cite{Al}.

\medskip
{\bf Proposition A } \label{teonV1V}
{\it
  Let $G=G(k_1,k_2,\dots,k_n,l_2,l_3,\dots,l_n)$ the Artin group with underlying graph
  
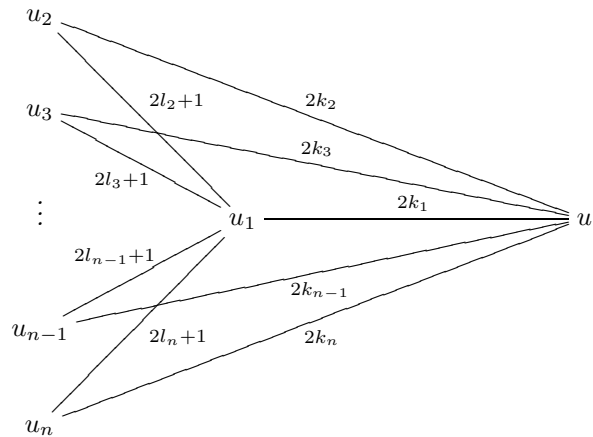
\begin{equation*}
  \xymatrix{
  u_2\ed[rrrrrrdd]^{2k_2}\ed[rrdd]^{2l_2+1} \\
  u_3\ed[rrrrrrd]^{2k_3}\ed[drr]_{2l_3+1}\\
  \vdots &&u_1\ed[rrrr]^{2k_1}                &&&& u\\
  u_{n-1}\ed[rrrrrru]_{2k_{n-1}}\ed[rru]^{2l_{n-1}+1}\\
  u_n\ed[rrrrrruu]_{2k_n}\ed[rruu]_{2l_n+1}
  }
  \end{equation*}
  such that $k_i,l_i$ are positive integers, with $k_i>1$ for all $i$ and $[\chi]\in\Hom(G,\real)$ such that $-1=\chi(u_1)=-\chi(u).$

  Besides, suppose
  \begin{equation}\label{eqhyp}
  \sum_{k_i \text{ is a potence of 2}}\frac{1}{2l_i+1}<1.
  \end{equation}

  Then $[\chi]\in\Sigma^1(G)^c$.}

\medskip
The proof of Proposition A is quite technical. It is given in section \ref{proofA}, where it is split in many small subsections. By Lemma \ref{corkerchifinger} Proposition A
is equivalent to ``$N = Ker(\chi)$ is not finitely generated'' and the idea of the proof of Proposition A is to find a quotient of $N$ that is not finitely generated.

Write $\ggoth$ for the  underlying graph of the group $G$ above. Using Proposition A we classify $\Sigma^1(G)$ for the class of groups considered in Proposition A.

\medskip
{\bf Theorem B}
  {\it Let $G = G(k_1,k_2,\dots,k_n,l_2,l_3,\dots,l_n)$ be an Artin group such that $k_i,l_i$ are positive integers with and $k_i>1$ for all $i$. Then
  $$\Sigma^1(G)=\{[\mu]\in S(G)\tq \mathcal{L}(\mu) \text{ is a connected dominant subgraph of $\mathcal{G}$} \}.$$
}

\subsection{On the $\Sigma$-invariants}

In this section we discuss several results from $\Sigma$-theory.  Let $G$ be a finitely generated group.
By definition $S(G)$ is the set of equivalence classes $[\chi]$ of non-zero real characters $\chi : G \to \mathbb{R}$ with respect to the equivalence relation $\sim$, where $\chi_1 \sim \chi_2$ precisely when there exists a positive real number $r$ such that $\chi_1 = r \chi_2$ i.e.  $[\chi] = R_{>0} \chi$. Thus $S(G)$ can be identified with the unit sphere $S^{n-1}$, where $n = \dim_{\mathbb{Q}} \ G/ G' \otimes_{\mathbb{Z}} \mathbb{Q} $. By definition for a finitely generated $\z G$-module $A$
$$
\Sigma^m(G, A) = \{ [\chi] \in S(G) \tq A \hbox{ has type } FP_m \hbox{ as left } \z[G_{\chi}]\hbox{-module} \},
$$
where $G_{\chi} = \{ g \in G  | \chi(g) \geq 0 \}$. The homotopical invariant $\Sigma^m(G)$ has a more complicated definition.  Here we will need only the case $m = 1$, where the situation is much simpler, since for any finitely generated group $G$ we have $\Sigma^1(G) = \Sigma^1(G, \z)$, where $\z $ is the trivial $\z G$-module.
For a subset $S \subseteq S(G)$ we write $S^c$ for the complement $S(G) \setminus S$.
The following result is a monoid version of the fact that a quotient of a finitely generated group is finitely generated.

\begin{lema} \label{corsigmaquocartin} Let $\pi:G \to \bar{G}$ be a group epimorphism and $\bar{\chi}$ be a non-trivial character of $\bar{G}$ such that $ \left[\bar{\chi}\circ\pi\right]\in\Sigma ^1(G)$. Then
  $\left[\bar{\chi}\right]\in\Sigma ^1\left(\bar{G}\right)$.
\end{lema}

The next result is one of the most important results on $\Sigma$-theory and one of the main motivations of the study of $\Sigma$-invariants.

\begin{teo} [\cite{BiRe}] \label{teosigma}
    Let $G$ be a group of homological type $FP_m$ and $H$ be a subgroup of $G$ containing  the commutator $[G,G]$. Then $H$ is of type $FP_m$ if and only if
    $S(G,H):=\{[\chi]\in S(G)\tq \chi(H)=0\}\subset \Sigma ^m(G,\mathbb{Z}) .$\end{teo}

Finally, the next result is an easy consequence of the application of Bieri-Renz's theorem on Artin groups.

\begin{lema}[\cite{AlKo}, Corollary 2.11]\label{corkerchifinger}
  Let $G$ be an Artin group and $\chi$ a discrete character of $G$. Then
  $[\chi]\in\Sigma^1(G) $ if and only if $ Ker(\chi) $  is finitely generated.
\end{lema}

\section{Proposition A implies Theorem B}

Note that by \cite{MeMnWy} if $G$ is an Artin group with underlying graph $\mathcal{G}$ and $\mu :  G \to \mathbb{R}$ a non-zero real character of $G$ such that  $\mathcal{L}(\mu)$ is a connected dominant subgraph of $\mathcal{G}$ then  $[\mu] \in \Sigma^1(G)$.
Observe that under the assumptions of Theorem B if $\mu : G \to \mathbb{R}$ is a non-zero homomorphism then $\mu(u_1) = \mu(u_2)=\cdots=\mu(u_n)$.
If the restriction of $\mu$ on all the vertices is non-zero and there isn't a dead edge then $\mathcal{L}(\mu) = \mathcal{G}$, so trivially $\mathcal{L}(\mu)$ is a connected dominant subgraph of $\mathcal{G}$ and hence $[\mu] \in \Sigma^1(G)$. If the restriction of $\mu$ on all the vertices is non-zero and there is a dead edge then $\mu(u) = 1 = - \mu(u_1) = - \mu(u_2) = \cdots = - \mu(u_n)$ or $\mu(u) = -1 = - \mu(u_1) = - \mu(u_2) =\cdots = - \mu(u_n)$. In this case by Proposition A we have that $[\mu] \notin \Sigma^1(G)$.

It remains to consider the following 2 cases : $\mu(u_1) = \mu(u_2) = \cdots = \mu(u_n) = 0$ or $\mu(u) = 0 $.
In both cases $\mathcal{L}(\mu)$ is a connected dominant subgraph of $\mathcal{G}$ hence $[\mu] \in \Sigma^1(G)$.

Recall that as stated above by \cite{MeMnWy} if $\mathcal{L}(\mu)$ is a connected dominant subgraph of $\mathcal{G}$ then  $[\mu] \in \Sigma^1(G)$.
By the last two paragraphs if $L(\mu)$ is not a connected dominant subgraph of $\mathcal{G}$ then either $\mu = \chi$ or $\mu = - \chi$, where $\chi$ is
the character considered in Proposition A. Then Lemma \ref{corkerchifinger} completes the proof.

\section{The proof of Proposition A} \label{proofA}

The proof of Proposition A is quite technical and is split in several subsections.

\subsection{An infinite presentation for $N = ker( \chi)$}

By definition $G$ has a presentation
\begin{align*}
G= \apres{u,u_1,u_2,\dots,u_n}{   &(uu_1)^{k_1}=(u_1u)^{k_1}, (uu_2)^{k_2}=(u_2u)^{k_2},\dots \\
                                    &\dots,(uu_n)^{k_n}=(u_nu)^{k_n}, (u_1u_2)^{l_2}u_1=(u_2u_1)^{l_2}u_2,\\
                                    &(u_1u_3)^{l_3}u_1=(u_3u_1)^{l_3}u_3,\dots,(u_1u_n)^{l_n}u_1=(u_nu_1)^{l_n}u_n}.
\end{align*}

Observe that if $\tilde{k}_i \geq 2$ is a divisor of $k_i$ for all $1 \leq i \leq n$ there is an epimorphism of
groups $$\pi : G(k_1,k_2,\dots,k_n,l_2,l_3,\dots,l_n) \to G(\tilde{k}_1,\tilde{k}_2,\dots,\tilde{k}_n,l_2,l_3,\dots,l_n)$$ induced by the identity on the vertices. Then, by Lemma  \ref{corsigmaquocartin}, to prove Proposition A we can work with $G(\tilde{k}_1,\tilde{k}_2,\dots,\tilde{k}_n,l_2,l_3,\dots,l_n)$ instead of\\  $G(k_1,k_2,\dots,k_n,l_2,l_3,\dots,l_n)$ i.e. substitute $k_i$ with $\tilde{k}_i$.
Thus we can assume that $k_i$ is prime for all $i$. For a technical reason, if the original $k_i$ is not a potence of $2$ we choose the new $k_i$ as a prime factor other than 2.

By Lemma \ref{corkerchifinger}, it is enough to prove that $N:=\ker (\chi)$ is not finitely generated. Suppose that $N$ is finitely generated.
Our strategy will be to find a finite sequence of quotients of $N$ in such a way that the last one is not finitely generated, leading to a contradiction.
Define
$$x_{0,i}:=uu_i,\quad i=1,\dots,n,\quad y_0:=x_{0,1}$$
 which means $u_i=u^{-1}x_{0,i}$ for all $i$.
  Then $N=\left\langle \{x_{0,i}\}_{1\leq i\leq n}\right\rangle ^G$.
  Define also
  $$x_{j,i}:=x_{0,i}^{u^j}\quad i=1,\dots,n,\quad j\in\z,$$
   which are generators of $N$ as a subgroup of $G$. We will also use the notation $\quad y_j:=x_{j,1}.$

Conjugating by powers of $u$, we can rewrite the relations of $G$ in the following way:

For $i=1,2,\dots,n$,
\begin{equation}\label{eqnV1Vrelpar}
 (u_iu)^{k_i}=(uu_i)^{k_i} \Leftrightarrow x_{j,i}^{k_i}=x_{0,i}^{k_i}, \quad \forall j\in\z.
 \end{equation}

For $i=2,3,\dots,n$,
 \begin{multline}\label{eqnV1Vrelimpar}
 (u_1u_i)^{l_i}u_1=(u_iu_1)^{l_i}u_i \Leftrightarrow \\
 x_{j,i}= y_{j-1}^{-1}x_{j-2,i}^{-1}y_{j-3}^{-1}x_{j-4,i}^{-1} \cdots y_{j-2l_i+3}^{-1}x_{j-2l_i+2,i}^{-1}y_{j-2l_i+1}^{-1}x_{j-2l_i,i}^{-1} \cdot \\ \cdot y_{j-2l_i}x_{j-2l_i+1,i}y_{j-2l_i+2}x_{j-2l_i+3,i} \cdots x_{j-3,i}y_{j-2}x_{j-1,i}y_j,  \quad \forall j\in\z.
 \end{multline}

Therefore $N$ has the following infinite  presentation
 $$N=\apres{\{x_{j,i}\}_{0\leq i \leq n, j\in \z}}{\eqref{eqnV1Vrelpar}, \eqref{eqnV1Vrelimpar}, 1\leq i\leq n, j\in \z}.$$

\subsection{A special quotient $\overline{N}$ of $N$}

Define the following quotient of $N$:

\begin{equation*}
\overline{N}:=\apres{N}{x_{0,i}^{k_i}=1, 1\leq i \leq n}.
\end{equation*}

Note that, in $\overline{N}$,

$$x_{j,i}^{k_i}=1,\quad \pt 1\leq i\leq n,\quad \pt j\in\z, $$

by \eqref{eqnV1Vrelpar}.

Define the groups

\begin{multline*}
K_i:=\apres{x_{0,i},x_{1,i},\cdots,x_{2l_i-1,i}}{x_{j,i}^{k_i}=1, j=0,1,\cdots,2l_i-1} \\
= \ast_{0\leq j\leq 2l_i-1}\apres{ x_{j,i} } { x_{j,i}^{k_i}=1 } \simeq \ast_{0\leq j\leq 2l_i-1} \z_{k_i},\quad 2\leq i\leq n,
\end{multline*}

$$K:=K_1=\apres{\{y_j\}_{j\in\z}}{y_j^{k_1}=1, j\in\z} = \ast_{j\in\z}\apres{y_j}{y_j^{k_1}=1} \simeq \ast_{j\in\z} \z_{k_1},$$

$$M:= \ast_{1\leq i\leq n} K_i, \quad M_i:= K \ast K_i, \quad \text{for } 2\leq i\leq n,$$

$$D:= \times_{2\leq i\leq n} \z_{k_i}, \quad A:=K\ast D .$$

Note that by \eqref{eqnV1Vrelimpar} $\{x_{j,i}\}_{j\in\z}$ may be seen as a subset of $M_i$ for each $i=2,\dots,n$. So we can think of $\overline{N}$ as a quotient of $M$. As a quotient of $N$, we have that $\overline{N}$ is also finitely generated.

\subsection{A technical lemma}

By \eqref{eqnV1Vrelimpar} we can define elements $\{x_{j,i}\}_{j\in \z}\subset M_i$. The following is a technical lemma about these elements.

\begin{lema}\label{lemanV1Vdecomp}
For each $2\leq i\leq n$ and for all $j\in\z$, there are unique $\tilde{x}_{j,i}\in K_i$ and $v_{j,i}\in K^{K_i}$ such that
$$x_{j,i}=\tilde{x}_{j,i}v_{j,i}.$$
 Note that if $j=0,1,\cdots,2l_i-1$ we have $\tilde{x}_{j,i}=x_{j,i}$ so $v_{j,i}=1$.

Besides, for each $2\leq i\leq n$,

\begin{multline}\label{eqdecomp}
\tilde{x}_{j,i}= \xtil_{j-2,i}^{-1}\xtil_{j-4,i}^{-1} \cdots \xtil_{j-2l_i+2,i}^{-1}\xtil_{j-2l_i,i}^{-1}\xtil_{j-2l_i+1,i}\xtil_{j-2l_i+3,i} \cdots \xtil_{j-3,i}\xtil_{j-1,i},\\ \forall j\in\z
\end{multline}

\begin{equation}\label{eqnV1Vlemavji+}
v_{j,i}y_j^{-1}\in \left\langle \{y_k^{K_1}\}_{0\leq k<j} \right\rangle, \quad \text{for } j\geq 2l_i
\end{equation}

\begin{equation}\label{eqnV1Vlemavji-}
v_{j,i}^{-1}y_j^{\xtil_{j,i}}\in \left\langle \{y_k^{K_i}\}_{j+1\leq k\leq 2l_i-1} \right\rangle \quad \text{for } j<0.
\end{equation}
\end{lema}

\begin{proof}
  Choose $2\leq i\leq n$. The first part is an easy consequence of normal form theorem for free products. We obtain \eqref{eqdecomp} by applying conjugations on \eqref{eqnV1Vrelimpar} and then using the uniqueness of the first part.

 Suppose $j\geq 2l_i$. To ease notation, let $\zeta_{k,i}$ be the product in $K_i$ of the $k$ last factors of the right side of  \eqref{eqdecomp}, so $\xtil_{j,i}=\zeta_{2l_i,i}$. Substituting $x_{k,i}=\xtil_{k,i}v_{k,i}$ in \eqref{eqnV1Vrelimpar} and conjugating the factors, from right to left, in way that we move the $\xtil_{k,i}$ to the left, we obtain

    \begin{multline*}
      x_{j,i}=\zeta_{2l_i,i} \left(y_{j-1}^{\zeta_{2l_i,i}}\right)^{-1} \left(v_{j-2,i}^{\zeta_{2l_i,i}}\right)^{-1} \left(y_{j-3}^{\zeta_{2l_i-1,i}}\right)^{-1} \left(v_{j-4,i}^{\zeta_{2l_i-1,i}}\right)^{-1} \cdots \\ \cdots \left(y_{j-2l_i+3}^{\zeta_{l_i+2,i}}\right)^{-1} \left(v_{j-2l_i+2,i}^{\zeta_{l_i+2,i}}\right)^{-1} \left(y_{j-2l_i+1}^{\zeta_{l_i+1,i}}\right)^{-1} \left(v_{j-2l_i}^{\zeta_{l_i+1,i}}\right)^{-1} \cdot \\ \cdot \left(y_{j-2l_i}^{\zeta_{l_i,i}}\right) \left(v_{j-2l_i+1,i}^{\zeta_{l_i-1,i}}\right) \left(y_{j-2l_i+2}^{\zeta_{l_i-1,i}}\right) \left(v_{j-2l_i+3,i}^{\zeta_{l_i-2,i}}\right) \cdots \\ \cdots \left(v_{j-3,i}^{\zeta_{1,i}}\right) \left(y_{j-2}^{\zeta_{1,i}}\right) \left(v_{j-1,i}\right) \left(y_j\right).
    \end{multline*}

    Then the uniqueness of the decomposition gives us the desired result for $j\geq 2l_i$.

    For $j<0$ we use a similar argument. Translating the indexes in \eqref{eqdecomp} from $j$ to $j+2l_i$ and reorganizing the equation we obtain

    \begin{multline}\label{eqlemanV1Vxtilneg}
      \xtil_{j,i}=\xtil_{j+1,i}\xtil_{j+3,i}\cdots\xtil_{j+2l_i-3,i}\xtil_{j+2l_i-1,i}\xtil_{j+2l_i,i}^{-1}\xtil_{j+2l_i-2}^{-1}\cdots\xtil_{j+4,i}^{-1}\xtil_{j+2,i}^{-1},\\ \pt j\in\z.
    \end{multline}

    By doing the same to \eqref{eqnV1Vrelimpar} we obtain

    \begin{multline}\label{eqlemanV1Vxneg}
    x_{j,i}=y_{j}x_{j+1,i}y_{j+2}x_{j+3,i}\cdots x_{j+2l_i-3,i}y_{j+2l_i-2}x_{j+2l_i-1,i}y_{j+2l_i} \cdot \\ \cdot x_{j+2l_i,i}^{-1}y_{j+2l_i-1}^{-1}x_{j+2l_i-2,i}^{-1}y_{j+2l_i-3}^{-1}\cdots x_{j+4,i}^{-1}y_{j+3}^{-1}x_{j+2,i}^{-1}y_{j+1,i}^{-1},\quad \pt j\in\z.
    \end{multline}

    To ease notation, consider $\zeta'_{k,i}$ as being the product in $K_i$ of the $k$ last factors on the right side of  \eqref{eqlemanV1Vxtilneg}, so $\xtil_{j,i}=\zeta'_{2l_1,i}$.

    Then substituting $x_{k,i}=\xtil_{k,i}v_{k,i}$ in \eqref{eqlemanV1Vxneg} and conjugating the factors, from the right to the left, in a way that we move the  $\xtil_{k,i}$ to the left we obtain

    \begin{multline*}
    x_{j,i}=\left(\zeta_{2l_i,i}'\right) \left(y_{j}^{\zeta_{2l_i,i}'}\right) \left(v_{j+1,i}^{\zeta_{2l_i-1,i}'}\right) \left(y_{j+2}^{\zeta_{2l_i-1,i}'}\right) \left(v_{j+3,i}^{\zeta_{2l_i-2,i}'}\right) \cdots \\ \cdots \left(v_{j+2l_i-3,i}^{\zeta_{l_i+1,i}'}\right) \left(y_{j+2l_i-2}^{\zeta_{l_i+1,i}'}\right) \left(v_{j+2l_i-1,i}^{\zeta_{l_i,i}'}\right) \left(y_{j+2l_i}^{\zeta_{l_i,i}'}\right) \cdot \\ \cdot \left(v_{j+2l_i,i}^{\zeta_{l_i,i}'}\right)^{-1} \left(y_{j+2l_i-1}^{\zeta_{l_i-1,i}'}\right)^{-1} \left(v_{j+2l_i-2,i}^{\zeta_{l_i-1,i}'}\right)^{-1} \left(y_{j+2l_i-3}^{\zeta_{l_i-2,i}'}\right)^{-1} \cdots \\ \cdots \left(v_{j+4,i}^{\zeta_{2,i}'}\right)^{-1} \left(y_{j+3}^{\zeta_{1,i}'}\right)^{-1} \left(v_{j+2,i}^{\zeta_{1,i}'}\right)^{-1} \left(y_{j+1}\right)^{-1}.
    \end{multline*}

    The result then follows from the uniqueness of the decomposition.
\end{proof}

\subsection{The epimorphism $\theta$}

We will now define  a epimorphism of groups
$\theta: M=K\ast \left(\ast_{2\leq i\leq n} K_i\right) \onto A = K * D$ such that $\theta_{|K}=\Id_K$ and that projects each $K_i$ onto $\z_{k_i}$. To do that, for each $2\leq i\leq n$ define $\theta$ as follows:

\begin{equation*}\theta(x_{0,i})=\theta(x_{1,i})=\cdots=\theta(x_{2l_i-2,i})=\bar{1}\in\z_{k_i}, \quad \theta(x_{2l_i-1,i})=\bar{2}\in\z_{k_i}.\end{equation*}

By first part of lemma \ref{lemanV1Vdecomp}, we have

$$\theta(\tilde{x}_{j+2l_i,i})= -\theta(\tilde{x}_{j,i})+\theta(\tilde{x}_{j+1,i})-\theta(\tilde{x}_{j+2,i})+\cdots+\theta(\tilde{x}_{j+2l_i-1,i}) \quad \forall j\in\z$$

which follows

\begin{multline} \label{eqnV1Vthetax}
\theta(\xtil _{2l_i,i})=\bar{1}, \quad \theta(\xtil _{2l_i+1,i})=\theta(\xtil _{2l_i+2,i})= \cdots =\theta(\xtil _{4l_i-1,i})=-\bar{1}, \\ \theta(\xtil_{4l_i,i})=-\bar{2},\quad \theta(\xtil _{4l_i+1,i})=-\bar{1},\\
 \quad \theta(\xtil _{4l_i+2,i})=\theta(\xtil _{4l_i+3,i})=\cdots =\theta(\xtil _{6l_i,i})=\bar{1}, \quad \theta(\xtil _{6l_i+1,i})=\bar{2}, \cdots
\end{multline}

and so

$$\theta(x_{j+4l_i+2,i})=\theta(x_{j,i}) \pt j\in\z.$$

By using the decomposition $x_{k,i}=\xtil_{k,i} v_{k,i}$, we obtain

$$
\theta(x_{j,i})^{k_i}= \underbrace{\left(\theta(\tilde{x}_{j,i})\theta(v_{j,i})\right)\left(\theta(\tilde{x}_{j,i})\theta(v_{j,i})\right)\cdots\left(\theta(\tilde{x}_{j,i})\theta(v_{j,i})\right)} _{k_i \text{ termos}},
$$

\begin{multline*}
\theta(x_{j,i})^{k_i}= \theta(\tilde{x}_{j,i})^{k_i} \left(\theta(v_{j,i})^{\theta(\tilde{x}_{j,i})^{k_i-1}}\right) \cdots \left(\theta(v_{j,i})^{\theta(\tilde{x}_{j,i})^2}\right) \left(\theta(v_{j,i})^{\theta(\tilde{x}_{j,i})}\right)\\ \left(\theta(v_{j,i})^1\right),
\end{multline*}

$$
\theta(x_{j,i})^{k_i}=\left(\theta(v_{j,i})^{\theta(\tilde{x}_{j,i})^{k_i-1}}\right) \cdots \left(\theta(v_{j,i})^{\theta(\tilde{x}_{j,i})^2}\right) \left(\theta(v_{j,i})^{\theta(\tilde{x}_{j,i})}\right) \left(\theta(v_{j,i})^1\right).
$$

Now define

$$x_i:=\bar{1}=\theta(x_{0,i})=\theta(\xtil_{0,i})\in\z_{k_i}.$$

Suppose $k_i\neq 2$. Then by \eqref{eqnV1Vthetax} $\theta(\xtil_{j,i})$ is a generator of $\z_{k_i}$ for all $j\in\z$. So, possibly after a permutation of exponents $\{x_i^{k_i-1},x_i^{k_i-2},\dots, x_i^2,x_i,1\}$ we have

\begin{equation}\label{eqnV1Vrelfinalx}
\theta(x_{j,i})^{k_i}=\left(\theta(v_{j,i})^{x_i^{k_i-1}}\right) \cdots \left(\theta(v_{j,i})^{x_i^2}\right) \left(\theta(v_{j,i})^{x_i}\right) \left(\theta(v_{j,i})^1\right).
\end{equation}

If  $k_i=2$ we have

\begin{equation}\label{eqnV1Vcasosxtili}
\theta(\xtil_{j,i})= \left\{    \begin{array}{ll}
                                e_A, & \text{if } j \equiv -2 \mod 2l_i+1 \\
                                x_i, & \text{else}
                            \end{array}                                        \right.
\end{equation}

where $e_A$ is the neutral element of $A$. So if $j \equiv -2 \mod 2l_i+1$ then
\begin{equation}\label{eqnV1Vrelfinalesp}
\theta(x_{j,i})^2=\theta(v_{j,i})^2.
\end{equation}
Note that this is the only case for which $\theta(\xtil_{j,i})=e_A$.

\subsection{One commutative diagram and the group $B$}

\noindent
Consider the canonical projection
$$\delta:A\onto \frac{A}{\left\langle \{\theta(x_{j,i})^{k_i}\}_{j\in\z,2\leq\i\leq n} \right\rangle ^A}=:\overline{A}$$
and the following commutative diagram of group homomorphisms
$$
\xymatrix{
  M\ar@{->>}[rrr]^{\theta}\ar@{->>}[d]_{\pi} &&& A \ar@{->>}[d]^{\delta} \\
  \overline{N}\ar@{->>}[rrr]_{\overline{\theta}}                &&& \overline{A}
  }
$$
where $\overline{\theta}$ is induced by $\theta$.
Note that since $\overline{N}$ is finitely generated $\overline{A}$ is finitely generated.

Note also that $K^A$ is a finite index subgroup of $A$, since
\begin{equation}\label{eqnV1VquocporK}
 \frac{A}{K^A}\simeq \times_{2\leq i\leq n}\z_{k_i}=D.
\end{equation}

Then $$B:=\delta(K^A)$$ is a finite index subgroup of $\overline{A}$, so $B$ is a finitely generated group. Note that $ \overline{A} / B \simeq  A / K^A\simeq D$, furthermore $A  = K^A \rtimes D$, hence $\overline{A} = B \rtimes D$.
As $x$ is a generator of $\z_{k_1}$ and $z$ is a generator of $\z_{k_3}$ as groups, then $\z[D]\simeq \z[x^{\pm 1},z^{\pm 1}]$. As $x_i$ is a generator of $\z_{k_i}$ for each $2\leq i\leq n$  then
$$\z_{k_1}[D]\simeq \z_{k_1}[x_2^{\pm 1},x_3^{\pm 1},\dots,x_n^{\pm 1}]$$
 as a $\z_{k_1}$- algebra.

\subsection{The abelianization $B^{ab}$ of $B$}

Consider the abelianization $B^{ab}$ of $B$. The $\z$-module $B^{ab}$ is finitely generated, because $B$ is a finitely generated group. We will think of $B^{ab}$ as a $\z [D]$-module, generated by the classes of
$$e_j:=\delta\theta (y_j),$$
where the action (denoted by $\circ$) of $D \simeq \overline{A} / B $ is induced by conjugation (denoted by $\circ$).

Define

$$p_{k}(a):=1+a+\cdots +a^{k-2}+a^{k-1}$$

and
$$\xhat_{j,i},\quad \ehat_j,\quad \vhat_{j,i}$$
   as being the imagens in $B^{ab}$ of   $$\delta\theta(\xtil_{j,i}),\quad e_j,\quad \delta\theta (v_{j,i}),$$
   respectively.

By \eqref{eqnV1Vrelfinalx} we have that, as $\z_{k_1}[D]$-modules,

\begin{equation}\label{eqnV1VBab}
B^{ab}\simeq \frac{\oplus_{j\in\z} e_j\left(\z_{k_1}[x_2^{\pm 1},x_3^{\pm 1},\dots,x_n^{\pm 1}]\right)}{I+I'},
\end{equation}

in such a way that $e_j\z_{k_1}[D]\simeq\z_{k_1}[D]$, $I$ is the $\z_{k_1}[D]$-submodule generated by

$$  \{\vhat_{j,i} \circ (p_{k_i}(x_i)) \tq \theta(\xtil_{j,i})\neq e_A\}$$

and $I'$ is the $\z_{k_1}[D]$-submodule generated by

$$\{k_i\vhat_{j,i} \tq \theta(\xtil_{j,i})=e_A\}=\{2\vhat_{j,i} \tq k_i=2, \theta(\xtil_{j,i})\}.$$

 \subsection{A contradiction : $B^{ab}$ is not finitely generated}

Without loss of generality, suppose the elements of $\{k_i\}_{2\leq i\leq n}$ which are potences of 2 are
$$k_2=k_3=\cdots=k_{m+1}=2,$$
so, by \eqref{eqhyp},

\begin{equation}\label{eqhyp2}
\sum_{2\leq i\leq m+1}\frac{1}{2l_i+1}<1.
\end{equation}

Define
$$L_i=\{j\in\z \tq j\equiv -2\mod 2l_i+1\}.$$

By \eqref{eqnV1Vrelfinalx}, \eqref{eqnV1Vcasosxtili} and \eqref{eqnV1Vrelfinalesp} $I$ is the $\z_{k_1}[D]$-submodule generated by

\begin{multline*}
\{\vhat_{j,i} \circ p_{k_i}(x_i),\\ \text{ for } (j,i)\in \left(\z \times \{m+2,m+3,\dots,n\}\right) \cup \left(\left(\z\setminus L_i\right) \times \{2,3,\dots, m+1\}\right)\}
\end{multline*}

and $I'$ is the $\z_{k_1}[D]$-submodule generated by
$$\left\{2\vhat_{j,i}, \text{ for } (j,i)\in L_i \times \{2,3,\dots,m+1\}\right\}.$$

Now define the $\z_{k_1}$-algebra

\begin{equation}\label{eqdefR}
R:=\frac{\z_{k_1}[x_2,x_3,\dots,x_n]}{\left( \{x_i+1\}_{2\leq i\leq m+1} \cup \{p_{k_i}(x_i)\}_{m+2\leq i\leq n}   \right)}\simeq \frac{\z_{k_1}[x_{m+2},x_{m+3},\dots,x_n]}{\left( \{p_{k_i}(x_i)\}_{m+2\leq i\leq n}   \right)},
\end{equation}

so that
$$x_2=x_3=\cdots=x_{m+1}=-1$$
 in $R$.

Note that $I'$ is contained in the $\z_{k_1}$-vector subspace of $R$ generated by
      $$\bigsqcup_{(j,i)\in L_i \times \{2,3,\dots,m+1\}}\vhat_{j,i}R,$$
       with the equality holding if $2\neq k_1$.

So the $\z_{k_1}$-vector space

$$W:=\frac{\oplus_{i\in\z} \ehat_iR}{\left\langle \bigsqcup_{(j,i)\in L_i \times \{2,3,\dots,m+1\}}\vhat_{j,i}R \right\rangle},\quad \text{onde } \ehat_i R\simeq R,$$

is a quotient of $B^{ab}$ hence finite dimensional.

Let

$$l:=\lcm (\{2l_i+1\}_{2\leq i\leq m+1})\geq 3.$$

Define also, for each $s\geq 1$, the set

$$\Lambda_s:=\{0,1,\dots,sl-1\}$$

and the following subspace of $W$:

$$E_s:=\frac{\oplus_{i\in\Lambda_s} \ehat_iR}{\left\langle \bigsqcup_{(j,i)\in L_i \times \{2,3,\dots,m+1\}}\vhat_{j,i}R \right\rangle \cap \left[ \oplus_{i\in\Lambda_s} \ehat_iR \right]}.$$

By lemma \ref{lemanV1Vdecomp},

\begin{itemize}
  \item
  $$\vhat_{j,i}-\ehat_j\in \oplus_{0\leq k< j}\ehat_k R,\quad \text{if } j\geq 2l_i;$$

  \item
  $$\vhat_{j,i}=0,\quad \text{if } 0\leq j\leq 2l_i-1;$$

  \item
  $$\vhat_{j,i}-\ehat_j\xhat_{j,i}\in \oplus _{j+1<k\leq 2l_i-1}\ehat_k R,\quad \text{if } j<0.$$
\end{itemize}

So

$$E_s=\frac{\oplus_{i\in\Lambda_s} \ehat_iR}{\left\langle \bigsqcup_{(j,i)\in \left[L_i\cap \Lambda_s\right] \times \{2,3,\dots,m+1\}}\vhat_{j,i}R \right\rangle }.$$

Note that by \eqref{eqdefR} we have the following equality of $\BZ_{k_1}$-vector spaces:

$$\left\langle \vhat_{j,i}R \right\rangle = \left\langle \{\vhat_{j,i}x_2^{j_2}x_3^{j_3}\cdots x_n^{j_n}\}_{0\leq j_i\leq k_i-2, 2\leq i\leq n} \right\rangle.$$

Then $\dim_{\z_{k_1}} \left\langle \vhat_{j,i}R \right\rangle=\dim R=\delta$, hence
\begin{align*}
\dim_{\z_{k_1}} \left\langle \bigsqcup_{(j,i)\in \left[L_i\cap \Lambda_s\right] \times \{2,3,\dots,m+1\}}\vhat_{j,i}R \right\rangle &\leq \delta \sum_{2\leq i\leq m+1} \sharp \left[ L_i\cap \Lambda_s \right]\\
&\leq \delta  \sum_{2\leq i\leq m+1}\frac{sl}{2l_i+1}\\
&\leq \delta sl \left( \sum_{2\leq i\leq m+1}\frac{1}{2l_i+1}\right).
\end{align*}

As $\dim_{\z_{k_1}}\oplus_{i\in\Lambda_s} e_iR=\delta sl$, we have

$$
\dim_{\z_{k_1}} E_s \geq \delta sl \left(1- \sum_{2\leq i\leq m+1}\frac{1}{2l_i+1}\right)>0,
$$

where the last inequality comes from \eqref{eqhyp2}. The above implies
$$\lim_{s\to \infty } \dim_{\z_{k_1}} E_s = \infty.$$
 As $E_s\leq W$ for each $s>0$, then $W$ is infinite dimensional, which is a contradiction.


\begin{thebibliography}{99}

\bibitem{Al}
K. Almeida. The $\Sigma^1$-invariant for Artin groups of circuit rank 2. In preparation.

\bibitem{AlKo} K. Almeida and D. H. Kochloukova. The $\Sigma^1$-invariant for Artin groups of circuit rank 1. \emph{Forum Mathematicum} \textbf{5} (2015), 2901-2925.

\bibitem{AlKo2} K. Almeida and D. H. Kochloukova. The $\Sigma^1$-invariant for some Artin groups of rank 3 presentation, \emph{Communications in Algebra} \textbf{43} (2015), 702-718.

\bibitem{Bi}
R. Bieri. \emph{Homological dimension of discrete groups} (Mathematics Department Queen Mary College, 1981).

\bibitem{BiGeKo}  R. Bieri and R. Geoghegan and D. Kochloukova. The sigma invariants of Thompson's group F. \emph{Groups Geom. Dyn.} \textbf{4 } (2010), no. 2, 263 - 273.

\bibitem{BiGr}  R. Bieri and J. R. Groves. The geometry of the set of characters induced by valuations. \emph{J. Reine Angew. Math.} \textbf{347} (1984), 168 - 195.

  \bibitem{BiNSt}
  R. Bieri and W. D. Neumann and R. Strebel. A geometric invariant of discrete groups. \emph{Invent. Math} \textbf{90} (1987), 451-477.


  \bibitem{BiRe}
  R. Bieri and B. Renz. Valuations on free resolutions and higher geometric invariants of groups. \emph{Comment. Math. Helv.} \textbf{63} (1988), 464-497.

  \bibitem{BiSt}
  R. Bieri and R. Strebel. Valuations and finitely presented metabelian groups. \emph{Proc. London Math. Soc.} \textbf{(3)41} (1980),
  439-464.



\bibitem{Ko}
D. H. Kochloukova. On the $\Sigma^2$-invariants of the generalised R. Thompson groups of type $F$.  \emph{Journal of Algebra} \textbf{371} (2012), 430-456.



\bibitem{Mn}
H. Meinert. The Homological Invariants for Metabelian Groups of Finite Prüfer Rank: A Proof of the $\Sigma^m$-Conjecture, \emph{Proc. London Math. Soc.} \textbf{(3) 72} (1996), n. 2, 385-424.

  \bibitem{MeMnWy}
  J. Meier and H. Meinert and L. VanWyk. On the $\Sigma$-invariants of Artin Groups. \emph{Topology and its Applications}, \textbf{110} (2001), 71-81.

  \bibitem{Re}
  B. Renz. Geometrische Invarianten Endlichkeitseigenschaften von Gruppen. Dissertation. Universitat Frankfurt a.M. (1988).



\end{thebibliography}
\end{document}